\documentclass[4pt]{article} 

\usepackage[utf8]{inputenc} 
\usepackage{geometry} 
\usepackage{graphicx} 

\usepackage[colorlinks,citecolor=red]{hyperref}
\usepackage{amsmath}
\usepackage{amsfonts}
\usepackage{dsfont}
\usepackage{mathrsfs}
\usepackage{latexsym}
\usepackage{graphicx}
\usepackage{amssymb}
\usepackage{float}
\usepackage{epsfig}
\usepackage{epstopdf}
\usepackage{color,xcolor}
\usepackage{booktabs} 
\usepackage{array} 
\usepackage{paralist} 
\usepackage{verbatim} 
\usepackage{subfig} 

\newtheorem{theorem}{Theorem}[section]
\newtheorem{lemma}[theorem]{Lemma}
\newtheorem{corollary}[theorem]{Corollary}
\newtheorem{proposition}[theorem]{Proposition}
\newtheorem{definition}[theorem]{Definition}

\newtheorem{remark}[theorem]{Remark}

\usepackage{fancyhdr} 
\usepackage{sectsty}
\allsectionsfont{\sffamily\mdseries\upshape} 

\usepackage[nottoc,notlof,notlot]{tocbibind} 
\usepackage[titles,subfigure]{tocloft} 

\newcommand{\la}{\langle}
\newcommand{\ra}{\rangle}

\title{Hamiltonian paths in the permutation digraphs $P(n,n-2)$}
\author{Jiaxin Guo$^{a}$, Ming Duan$^{a}$, Jie Xue$^{b}$\thanks{Emails:~jiaxin\_guo7@126.com (J. Guo),~mdscience@sina.com (M. Duan),~jie\_xue@126.com (J. Xue).}
\\
{\footnotesize $^a$ nformation Engineering University, 450001 Zhengzhou, China}\\
{\footnotesize $^b$ School of Mathematics and Statistics, Zhengzhou University, 450001 Zhengzhou, China}
}
\date{} 

\begin{document}
\maketitle

\begin{abstract}
For $1\leq k<n$, let $P(n,k)$ be the directed overlap graph whose vertices
are the $k$-permutations of $[n]$ and whose arcs are the
$(k+1)$-permutations. Isaak proved that $P(n,n-2)$ has no directed
Hamiltonian cycle for $n\geq4$ and asked whether it nevertheless has a
directed Hamiltonian path. We answer this question affirmatively by showing that $P(n,n-2)$ has a Hamiltonian path.
\bigskip

\noindent {\bf Key words:} $k$-permutation; Universal cycle; Eulerian tour; Laplacian matrix; Eigenvalue
\end{abstract}

\section{Introduction}

For $1\leq k<n$, the permutation digraph $P(n,k)$ has as its vertices the
ordered $k$-tuples of distinct elements of $[n]=\{1,\ldots,n\}$. There is an
arc
$
 (a_1,\ldots,a_k)\longrightarrow(a_2,\ldots,a_k,b)
$
whenever $a_1,\ldots,a_k,b$ are pairwise distinct. Thus the arcs are
naturally indexed by the $(k+1)$-permutations of $[n]$.

Universal cycles for combinatorial structures were introduced systematically
by Chung, Diaconis and Graham \cite{Chung1992}. Jackson \cite{Jackson1993}
constructed universal cycles for $k$-permutations for every $k\geq3$ and
$n\geq k+1$. The usual universal-cycle model, however, only requires every
new window of length $k$ to be a $k$-permutation. The symbol appended after
$(a_1,\ldots,a_k)$ may therefore equal the symbol $a_1$ that has just left
the window. In $P(n,k)$ the full string $a_1,\ldots,a_k,b$ must consist of
distinct symbols. Hamiltonian cycles and paths in $P(n,k)$ are consequently
more restricted than ordinary universal cycles for $k$-permutations.

Starling, Klerlein, Kier and Carr \cite{Starling2003} initiated the study of
Hamiltonicity in these digraphs. Isaak \cite{Isaak2006} proved that $P(n,k)$
contains a directed Hamiltonian cycle whenever $k\leq n-3$. He also proved
that
\[
 P(n,n-2)\ \text{has no directed Hamiltonian cycle for every }n\geq4.
\]
His proof identifies $P(n,n-2)$ with a two-generator Cayley digraph of $A_n$
and applies Rankin's theorem \cite{Rankin1948,Swan1999}. Isaak observed that
$P(4,2)$ has a Hamiltonian path and reported a computer verification for
$P(5,3)$, leaving the general path problem open.

Several related universal-cycle and Gray-code constructions have since been
developed. Ruskey and Williams \cite{Ruskey2010} gave a loopless construction
for shorthand universal cycles, and Holroyd, Ruskey and Williams
\cite{Holroyd2012} obtained further structural and algorithmic results.
Successor rules and cycle joining form a general framework for universal-cycle
generation \cite{Gabric2020}, while concatenation trees provide another way
to organize such joins \cite{Sawada2023}. Efficient constructions for
ordinary and shorthand permutation universal cycles appear in
\cite{Chang2025,Chang2026}. Other relevant models and problems can be found
in \cite{Jackson2009,Johnson2009,Gao2019,Wong2017}; see
\cite{Brunat1997,Curren1996,Witte1984} for background on permutation digraphs
and Hamiltonian cycles in Cayley graphs and digraphs.

The transitions used in those models do not resolve Isaak's question. Our
argument remains inside his two-generator Cayley digraph, except for one
temporary transition that is removed by cutting a cycle. The main result is
the following.

\begin{theorem}\label{thm:main}
For every integer $n\geq4$, the digraph $P(n,n-2)$ contains a directed
Hamiltonian path.
\end{theorem}

Let $\alpha,\beta\in A_n$ be the Cayley generators and put
$
 \rho=\alpha\beta^{-1}=(1\ n\ n-1).$
The orbits of right multiplication by $\alpha$ are the left
$\la\alpha\ra$-cosets. Every left $\la\rho\ra$-coset contains three elements
belonging to three distinct $\la\alpha\ra$-cosets. These incidences define a
$3$-uniform hypergraph $\mathcal H_n$. Switching all three successors on a
left $\la\rho\ra$-coset from $\alpha$ to $\beta$ merges three cycles whenever
the three sources lie in different cycles. A spanning loose hyperforest of
$\mathcal H_n$ with two components therefore gives a directed cycle cover
with two cycles. A crossing hyperedge permits a splice with one invalid
transition; deleting this transition leaves a Hamiltonian path.

The required two-component hyperforest is obtained from a root-deleted
spanning hypertree theorem for permutation hypergraphs generated by the
triangles of a $2$-tree. For odd $n$ the needed subhypergraph comes from a
$2$-path. For even $n$ it comes from a fan $2$-tree, after which the
remaining layers are attached by a perfect matching in a $2$-regular
bipartite multigraph.

Section~2 introduces the Cayley-digraph model and the coset hypergraph
$\mathcal{H}_n$. Section~3 reduces the problem to constructing a
spanning loose hyperforest with two components and proves the required
hypertree theorem. Section~4 constructs such a hyperforest using a
$2$-path for odd $n$ and a fan $2$-tree together with a perfect
matching for even $n$. The case $n=4$ is immediate, and these results
complete the proof of Theorem~1.1.

\section{Preliminaries}\label{sec:cayley}

Permutations are functions on $[n]$ and are composed from right to left:
$(gh)(i)=g(h(i))$. We write a permutation $g$ in one-line notation as
$
 g=(g(1),g(2),\ldots,g(n)).
$
Thus right multiplication changes positions, whereas left multiplication
changes values. For a group $G$ and $S\subseteq G$, the right Cayley digraph
$\operatorname{Cay}(G,S)$ has arcs
$
 g\longrightarrow gs\qquad(g\in G,\ s\in S).
$
This convention is convenient for overlap digraphs, since a right multiplier
can shift the entries of a one-line permutation.

For even $n$, define
\begin{equation}\label{eq:generators-even}
 \alpha=(1\ 2\ \cdots\ n-1),\qquad
 \beta=(1\ 2\ \cdots\ n-2\ n),
\end{equation}
and for odd $n$, define
\begin{equation}\label{eq:generators-odd}
 \alpha=(1\ 2\ \cdots\ n),\qquad
 \beta=(1\ 2\ \cdots\ n-2\ n\ n-1).
\end{equation}
When $n$ is even, both cycles in \eqref{eq:generators-even} have length
$n-1$; when $n$ is odd, both cycles in \eqref{eq:generators-odd} have length
$n$. In either case their lengths are odd, so $\alpha$ and $\beta$ are even
permutations.

Every vertex $v=(a_1,\ldots,a_{n-2})$ of $P(n,n-2)$ has two missing symbols,
say $u$ and $v'$. The two completions
$
 (a_1,\ldots,a_{n-2},u,v')
\text{ and }
 (a_1,\ldots,a_{n-2},v',u)
$
differ by a transposition. Exactly one is even. Denote this unique even
completion by $\widehat v$.

\begin{lemma}\label{lem:cayley-isomorphism}
The completion map $v\mapsto\widehat v$ is a digraph isomorphism
\begin{equation}\label{eq:cayley-isomorphism}
 P(n,n-2)\simeq\operatorname{Cay}(A_n,\{\alpha,\beta\}).
\end{equation}
\end{lemma}

\begin{proof}
The map is injective by construction. Both sets have cardinality $n!/2$, so
it is bijective. Write
$
 \widehat v=(a_1,\ldots,a_{n-2},u,v').
$
For either parity of $n$, the definitions above give
\[
\begin{aligned}
 (\widehat v\alpha)(1,\ldots,n-2)
   &=(a_2,\ldots,a_{n-2},u),\\
 (\widehat v\beta)(1,\ldots,n-2)
   &=(a_2,\ldots,a_{n-2},v').
\end{aligned}
\]
These are exactly the two out-neighbours of $v$ in $P(n,n-2)$. Since
$\widehat v$, $\alpha$ and $\beta$ are even, the products are precisely the
even completions of those out-neighbours. Hence the map preserves and
reflects all arcs.
\end{proof}

The construction is a regular group action: the vertex set is identified
with $A_n$, and the two overlap moves are induced by right multiplication by
$\alpha$ and $\beta$. No linear representation theory is used; the relevant
group representation is the concrete Cayley action together with the coset
decompositions associated with two cyclic subgroups.

Put
\begin{equation}\label{eq:rho}
 \rho=\alpha\beta^{-1}=(1\ n\ n-1).
\end{equation}
Consequently,
\begin{equation}\label{eq:relations}
 \rho^3=1,\qquad \alpha=\rho\beta,\qquad \beta=\rho^{-1}\alpha.
\end{equation}
The last identity turns an $\alpha$-successor into a $\beta$-successor after
moving the source once around a $\rho$-orbit.

Recall that the adjacent $3$-cycles
$
 (1\ 2\ 3),(2\ 3\ 4),\ldots,(n-2\ n-1\ n)
$
generate $A_n$.

\begin{lemma}\label{lem:generation}
For every $n\geq4$,
$
 \la\alpha,\beta\ra=\la\alpha,\rho\ra=A_n.
$
\end{lemma}

\begin{proof}
The first equality follows from \eqref{eq:relations}. Suppose first that $n$
is odd. For $1\leq i\leq n-2$,
$
 \alpha^{i+1}\rho\alpha^{-(i+1)}
   =(i\ i+2\ i+1),
$
the inverse of $(i\ i+1\ i+2)$. Hence $\la\alpha,\rho\ra=A_n$.

Now let $n$ be even and put $m=n-1$. For $1\leq i\leq m-1$, set
$
 \delta_i=\alpha^i\rho\alpha^{-i}=(i+1\ n\ i).
$
For $1\leq i\leq m-2$, direct multiplication gives
$
 \delta_i\delta_{i+1}=(i\ i+1\ i+2).
$
Thus the generated subgroup contains the natural copy of $A_m$ fixing $n$.
It also contains $\rho=(1\ n\ m)$. Since $A_n$ is generated by $A_m$ and
any $3$-cycle containing $n$ and two points of $[m]$, the result follows.
\end{proof}

The permutation $g\mapsto g\alpha$ decomposes $A_n$ into directed cycles.
The orbit through $g$ is the left coset
$
 g\la\alpha\ra
   =\{g,g\alpha,\ldots,g\alpha^{|\alpha|-1}\}.
$
The orbit of $g$ under right multiplication by $\rho$ is
$g\la\rho\ra=\{g,g\rho,g\rho^2\}$. Write
$
 \mathcal A_n=A_n/\la\alpha\ra$ and $
 \mathcal R_n=A_n/\la\rho\ra.
$

\begin{lemma}\label{lem:trivial-intersection}
For every $n\geq4$,
$
 \la\alpha\ra\cap\la\rho\ra=\{1\}.
$
Consequently, each coset in $\mathcal R_n$ meets exactly three distinct
cosets in $\mathcal A_n$, in one element each.
\end{lemma}

\begin{proof}
If $n$ is even, every element of $\la\alpha\ra$ fixes $n$, whereas $\rho$ and
$\rho^2$ move $n$. If $n$ is odd, $\alpha$ is an $n$-cycle. A nonidentity
power of $\alpha$ having order $3$, when such a power exists, is a product of
$n/3$ disjoint $3$-cycles. Since $n\geq5$, it cannot be the single
$3$-cycle $\rho$ or $\rho^2$. The intersection is trivial.

If two elements of $g\la\rho\ra$ belonged to the same
$\la\alpha\ra$-coset, their quotient would be a nonidentity element of the
intersection, a contradiction.
\end{proof}

\begin{definition}\label{def:Hn}
The coset hypergraph $\mathcal H_n$ is the $3$-uniform hypergraph with vertex
set $\mathcal A_n$ and hyperedge set $\mathcal R_n$, where incidence means
nonempty intersection of cosets.
\end{definition}

If $R=x\la\rho\ra$, its three incident vertices are
$
 x\la\alpha\ra,
 x\rho\la\alpha\ra$ and $
 x\rho^2\la\alpha\ra.
$
Lemma~\ref{lem:trivial-intersection} shows both that these vertices are
distinct and that the group element at each incidence is unique.

\begin{lemma}\label{lem:H-connected}
The hypergraph $\mathcal H_n$ is connected.
\end{lemma}

\begin{proof}
By Lemma~\ref{lem:generation}, every element of $A_n$ is a word in
$\alpha^{\pm1}$ and $\rho^{\pm1}$. Right multiplication by $\alpha$ stays
inside one vertex of $\mathcal H_n$, while right multiplication by $\rho$
moves between vertices incident with a common hyperedge. Projecting a word
to the $\la\alpha\ra$-cosets gives a walk between any two vertices.
\end{proof}

The incidence graph of a hypergraph is the bipartite graph whose parts are
its vertices and hyperedges. A $3$-uniform component is a \emph{loose
hypertree} if its incidence graph is a tree; an isolated vertex is also a
loose hypertree. A \emph{loose hyperforest} is a vertex-disjoint union of
loose hypertrees. Equivalently, the hyperedges in each nontrivial component
can be ordered so that the first has three vertices and every later hyperedge
meets the preceding vertex union in exactly one vertex and introduces two
new vertices. A spanning loose hyperforest with $N$ vertices and $c$
components has $(N-c)/2$ hyperedges.

\section{Cycle covers and hypergraphs}\label{sec:reduction}

Start with the successor permutation
$
 s_0(g)=g\alpha~(g\in A_n).
$
Its directed cycles are the cosets in $\mathcal A_n$.

\begin{lemma}\label{lem:switch}
Let $R=x\la\rho\ra=\{x,x\rho,x\rho^2\}$. Replace
$g\mapsto g\alpha$ by $g\mapsto g\beta$ for all $g\in R$. The resulting
successor map is again a permutation of $A_n$. If the three elements of $R$
lie in three distinct directed cycles before the switch, the switch merges
those cycles into one.
\end{lemma}

\begin{proof}
Since $\alpha=\rho\beta$,
$
 \{x\rho^i\alpha:0\leq i\leq2\}
   =\{x\rho^i\beta:0\leq i\leq2\}.
$
Thus the operation cyclically permutes three successor targets and leaves
every indegree equal to one. If the sources belong to distinct cycles, this
cyclic permutation of their targets joins those cycles into one.
\end{proof}

\begin{proposition}\label{prop:parity}
Every directed cycle cover of
$\operatorname{Cay}(A_n,\{\alpha,\beta\})$ has an even number of cycles.
In particular, the digraph has no directed Hamiltonian cycle.
\end{proposition}

\begin{proof}
Let $s$ be a successor permutation with
$s(g)\in\{g\alpha,g\beta\}$ for every $g\in A_n$. Fix
$R=\{x,x\rho,x\rho^2\}$ and put $t_i=x\rho^i\beta$, with indices modulo
$3$. The source $x\rho^i$ is sent to $t_i$ by a $\beta$-choice and to
$t_{i+1}$ by an $\alpha$-choice. These sources map bijectively to the three
targets only when all three choices are $\alpha$ or all three are $\beta$.
Any mixed choice repeats one target and omits another. Thus every cycle cover
is obtained from $s_0$ by switching a set of whole
$\la\rho\ra$-cosets.

Each switch is a $3$-cycle on successor targets and is even. The sign of the
successor permutation is unchanged. A permutation of $N$ points with $c$
cycles has sign $(-1)^{N-c}$, so the parity of $c$ is unchanged. The initial
number of cycles is
\begin{equation}\label{eq:qn}
 q_n=|A_n:\la\alpha\ra|=
 \begin{cases}
  (n-1)!/2,& n\ \text{odd},\\[2mm]
  n(n-2)!/2,& n\ \text{even}.
 \end{cases}
\end{equation}
This is even for odd $n\geq5$, for $n=4$, and for even $n\geq6$.
\end{proof}

Distinct cosets in $\mathcal R_n$ are disjoint, so switches on different
hyperedges act on disjoint sets of sources.

\begin{proposition}\label{prop:forest-cover}
Let $F$ be a spanning loose hyperforest of $\mathcal H_n$ with $c$
components. Switching every hyperedge of $F$ produces a directed cycle cover
with exactly $c$ cycles. The elements in the $\la\alpha\ra$-cosets of one
component of $F$ form one directed cycle.
\end{proposition}

\begin{proof}
Process each nontrivial component in a loose-hypertree order. The first
hyperedge meets three distinct $\alpha$-cycles and merges them by
Lemma~\ref{lem:switch}. Each later hyperedge meets one previously merged
cycle and two untouched $\alpha$-cycles, so another switch merges those
three cycles. Different components use disjoint vertices, and isolated
vertices remain single $\alpha$-cycles.
\end{proof}

\begin{lemma}\label{lem:splice}
Suppose $F$ is a spanning loose hyperforest of $\mathcal H_n$ with two
components, and let $s$ be the two-cycle successor permutation obtained in
Proposition~\ref{prop:forest-cover}. Two successors can be exchanged so that
the resulting permutation has one cycle and exactly one transition that is
not an arc of $\operatorname{Cay}(A_n,\{\alpha,\beta\})$.
\end{lemma}

\begin{proof}
By Lemma~\ref{lem:H-connected}, some hyperedge
$R=x\la\rho\ra$ is incident with both components of $F$. It does not belong
to $F$; moreover, none of its sources was switched because the
$\la\rho\ra$-cosets partition $A_n$. All three retain their
$\alpha$-successors.

Choose two elements in different directed cycles and, after cyclically
renaming $R$, take them to be $x$ and $x\rho$. Exchanging their targets
merges the two cycles. The transitions on $R$ are now
\[
 x\longmapsto x\rho\alpha,\qquad
 x\rho\longmapsto x\alpha,\qquad
 x\rho^2\longmapsto x\rho^2\alpha.
\]
The second is a $\beta$-transition because
$(x\rho)\beta=x\rho\rho^{-1}\alpha=x\alpha$, and the third remains an
$\alpha$-transition. The allowed targets from $x$ are $x\alpha$ and
$x\beta=x\rho^2\alpha$, so the first transition is the unique invalid one.
\end{proof}
Apply Proposition~\ref{prop:forest-cover} and Lemma~\ref{lem:splice}. Delete
the unique invalid transition from the resulting cycle and use the
isomorphism \eqref{eq:cayley-isomorphism}.
We obtain the following theorem.
\begin{theorem}\label{thm:reduction}
If $\mathcal H_n$ contains a spanning loose hyperforest with two components,
then $P(n,n-2)$ contains a directed Hamiltonian path.
\end{theorem}

A $2$-tree is obtained from a triangle by repeatedly adding a new vertex
adjacent to the two ends of an existing edge. Equivalently, every $2$-tree
of order greater than three has a simplicial vertex of degree two whose
deletion is again a $2$-tree. A $2$-tree on $m$ vertices has $m-2$
triangles.

Let $X$ and $Y$ be linearly ordered sets of cardinality $m$, and let
$\Omega$ be either parity class of bijections from $X$ to $Y$. Let $T$ be a
$2$-tree on $X$. For every triangle $Q$ of $T$, choose one of the two
$3$-cycles supported by $Q$ and call it $\tau_Q$.

\begin{definition}\label{def:permutation-hypergraph}
The permutation hypergraph $K(T,\Omega)$ has vertex set $\Omega$ and
hyperedges
\begin{equation}\label{eq:permutation-edge}
 E_Q(\pi)=\{\pi,\pi\tau_Q,\pi\tau_Q^2\}
\end{equation}
for every triangle $Q$ of $T$ and every $\pi\in\Omega$.
\end{definition}

The orientation of $\tau_Q$ is immaterial, and all three vertices in
\eqref{eq:permutation-edge} lie in $\Omega$ because a $3$-cycle is even.

\begin{theorem}\label{thm:root-deleted}
Let $m\geq4$, let $T$ be a $2$-tree on $X$, let $\Omega$ be a parity class
of bijections $X\to Y$, and let $r\in\Omega$. Then $K(T,\Omega)-r$ contains
a spanning loose hypertree. Equivalently, $K(T,\Omega)$ contains a spanning
loose hyperforest whose components are the isolated vertex $r$ and a loose
hypertree on $\Omega\setminus\{r\}$.
\end{theorem}

\begin{proof}
We use induction on $m$. For $m=4$, relabeling $X$ and $Y$ and
left-composing all bijections with $r^{-1}$ reduces the problem to
$X=Y=[4]$, $\Omega=A_4$, $r=1234$ in one-line notation, and triangles
$\{1,2,3\}$ and $\{2,3,4\}$. The following five hyperedges belong to
$K(T,A_4)$:
\[
\begin{array}{ll}
 E_1=\{1342,3412,4132\},&
 E_2=\{1423,2143,4213\},\\
 E_3=\{2143,2314,2431\},&
 E_4=\{2431,3241,4321\},\\
 E_5=\{3124,3241,3412\}.&
\end{array}
\]
In the order $E_2,E_3,E_4,E_5,E_1$, each edge after the first meets the
preceding union in one vertex. Their union is $A_4\setminus\{1234\}$, so
they form the required loose hypertree.

Let $m\geq5$, and choose a simplicial vertex $x$ of $T$ such that $T-x$ is a
$2$-tree. Write $\{x,a,b\}$ for the unique triangle containing $x$ and take
its generator to be
$
 \tau=(x\ a\ b).
$
For $j\in Y$, define the layer
$
 L_j=\{\pi\in\Omega:\pi(x)=j\}.
$
Every triangle generator other than $\tau$ fixes $x$. After restriction to
$X\setminus\{x\}$, the hypergraph induced on $L_j$ is a permutation
hypergraph associated with $T-x$ and one parity class of bijections
$X\setminus\{x\}\to Y\setminus\{j\}$.

Put $c=r(x)$ and $J=Y\setminus\{c\}$. Choose a cyclic permutation $\sigma$
of $J$ such that
\begin{equation}\label{eq:sigma-condition}
 \sigma(r(b))\neq r(a).
\end{equation}
This is possible because $|J|=m-1\geq4$: a cyclic order can avoid one
prescribed directed adjacency. Set $r_c=r$. For every $j\in J$, choose
$r_j\in L_j$ satisfying
\begin{equation}\label{eq:rj}
 r_j(x)=j,\qquad r_j(a)=c,\qquad r_j(b)=\sigma(j).
\end{equation}
The three prescribed images are distinct. At least two positions and two
values remain, so the parity can be adjusted by interchanging two unused
images. Hence $r_j$ can always be chosen in $\Omega$.

By induction, each layer $L_j$ contains a loose hypertree $B_j$ spanning
$L_j\setminus\{r_j\}$. Initially that layer contributes the two components
$B_j$ and $R_j=\{r_j\}$. For $j\in J$, add
\begin{equation}\label{eq:Fj}
 F_j=\{r_j,r_j\tau,r_j\tau^2\}.
\end{equation}
The conditions in \eqref{eq:rj} imply
$
 r_j\tau\in L_c,$ and $ r_j\tau^2\in L_{\sigma(j)}.
$
Moreover, $r_j\tau\neq r_c$: equality would give
$j=r(b)$ and $\sigma(j)=r(a)$, contrary to
\eqref{eq:sigma-condition}. Also $r_j\tau^2\neq r_{\sigma(j)}$, since
$
 (r_j\tau^2)(a)=r_j(x)=j\neq c=r_{\sigma(j)}(a).
$
Thus $F_j$ meets the three components $R_j$, $B_c$, and
$B_{\sigma(j)}$.

Write the cycle of $\sigma$ as $(j_1\,j_2\,\cdots\,j_q)$. Add
$F_{j_1},F_{j_2},\ldots,F_{j_q}$ in this order. Before $F_{j_t}$ is
added, the growing component containing $B_c$ also contains
$B_{j_2},\ldots,B_{j_t}$ and $R_{j_1},\ldots,R_{j_{t-1}}$. The singleton
$R_{j_t}$ and the hypertree $B_{j_{t+1}}$, with indices read cyclically, are
still separate. Hence each new hyperedge joins three distinct
incidence-tree components and creates no incidence cycle. At the end every
vertex except $r_c=r$ belongs to one loose hypertree, while $r$ remains
isolated.
\end{proof}

We also require a form in which one complete layer is omitted.

\begin{corollary}\label{cor:layer-deletion}
Retain the notation of Theorem~\ref{thm:root-deleted}, let $m\geq5$, suppose
$x$ is simplicial in $T$, and fix $s\in Y$. The hypergraph induced on
$
 \Omega\setminus L_s$ and $
 L_s=\{\pi\in\Omega:\pi(x)=s\},
$
contains a spanning loose hyperforest with two components.
\end{corollary}

\begin{proof}
Let $\{x,a,b\}$ be the unique triangle containing $x$ and put
$\tau=(x\ a\ b)$. Set $I=Y\setminus\{s\}$. Choose $c\in I$ and a cyclic
permutation $\sigma$ of $I\setminus\{c\}$. For every
$j\in I\setminus\{c\}$ choose $r_j\in L_j$ satisfying
$
 r_j(x)=j,$ and $ r_j(a)=c $ and $ r_j(b)=\sigma(j).
$
As above, the parity can be chosen as required. Choose
$
 r_c\in L_c\setminus
 \{r_j\tau:j\in I\setminus\{c\}\}.
$
This is possible because $|L_c|=(m-1)!/2>m-2$.

Apply Theorem~\ref{thm:root-deleted} within each layer $L_j$, $j\in I$, to
obtain a loose hypertree $B_j$ spanning $L_j\setminus\{r_j\}$. For each
$j\in I\setminus\{c\}$, the edge $F_j$ in \eqref{eq:Fj} meets the singleton
$R_j$, the hypertree $B_c$, and the hypertree $B_{\sigma(j)}$. The choice of
$r_c$ ensures $r_j\tau\neq r_c$, and
$
 (r_j\tau^2)(a)=j\neq c=r_{\sigma(j)}(a).
$
Adding the $F_j$ in the cyclic order of $\sigma$ joins three incidence-tree
components at every step. The final components are the isolated vertex
$r_c$ and one loose hypertree containing every other vertex of
$\Omega\setminus L_s$.
\end{proof}

\begin{remark}\label{rem:deterministic}
The proof is constructive. Order bijections lexicographically. At every
stage choose the least admissible simplicial vertex, cyclic permutation, and
representatives $r_j$; in the base case, use the hyperedges in the order
$
E_2,E_3,E_4,E_5,E_1.
$ These choices make the recursion deterministic.
\end{remark}

\section{Proof of the main theorem}\label{sec:odd}

Assume that $n\geq5$ is odd and put $m=n-1$. Then
$\alpha=(1\ 2\ \cdots\ n)$. Every coset in $\mathcal A_n$ contains a unique
representative $g$ satisfying $g(n)=n$, since the value
$(g\alpha^t)(n)$ runs once through $[n]$. This representative is even, and
restriction to $[m]$ gives a bijection
\begin{equation}\label{eq:odd-vertex-map}
 \Phi:A_m\longrightarrow\mathcal A_n,\qquad
 \Phi(g)=g\la\alpha\ra,
\end{equation}
where $g\in A_m$ is regarded as fixing $n$.

For $1\leq i\leq m-2$, put
$
 \tau_i=(i\ i+1\ i+2).
$
A direct conjugation gives
\begin{equation}\label{eq:odd-conjugate}
 \alpha^{i+1}\rho\alpha^{-(i+1)}=\tau_i^{-1}.
\end{equation}
For $g\in A_m$, the coset $g\alpha^{i+1}\la\rho\ra$ is incident with the
three vertices whose standard representatives are $g,g\tau_i,g\tau_i^2$.

We verify that these incidences give an embedded permutation hypergraph.
For a $\tau_i$-orbit
$
 E_i(g)=\{g,g\tau_i,g\tau_i^2\},
$
define
$
 \Psi(E_i(g))=g\alpha^{i+1}\la\rho\ra.
$
This is well defined because
$g\tau_i^t\alpha^{i+1}=g\alpha^{i+1}\rho^{-t}$. Suppose
$
 g\alpha^{i+1}\la\rho\ra
   =h\alpha^{j+1}\la\rho\ra
$
for $g,h\in A_m$ and $1\leq i,j\leq m-2$. For some
$t\in\{0,1,2\}$,
$
 h=g\alpha^{i+1}\rho^t\alpha^{-(j+1)}.
$
Since $g$ and $h$ fix $n$, the permutation
$u=\alpha^{i+1}\rho^t\alpha^{-(j+1)}$ fixes $n$. Now
$
 \alpha^{-(j+1)}(n)=n-j-1\in\{2,\ldots,n-2\},
$
and this point is fixed by $\rho^t$. Thus
$u(n)=\alpha^{i+1}(n-j-1)$, which equals $n$ only if $i=j$.
Equation \eqref{eq:odd-conjugate} then gives $h=g\tau_i^{-t}$, so the two
abstract hyperedges coincide. Hence $\Phi$ and $\Psi$ form an incidence
isomorphism onto their image.

The triangles
\begin{equation}\label{eq:2path}
 \{1,2,3\},\{2,3,4\},\ldots,\{m-2,m-1,m\}
\end{equation}
form a $2$-path. The corresponding $K(T,A_m)$ is a spanning subhypergraph
of $\mathcal H_n$.

\begin{proposition}\label{prop:odd-forest}
If $n\geq5$ is odd, then $\mathcal H_n$ contains a spanning loose
hyperforest with two components.
\end{proposition}

\begin{proof}
Apply Theorem~\ref{thm:root-deleted} to the $2$-path in
\eqref{eq:2path}, with any prescribed root, and transfer the resulting
root-isolated hyperforest through the incidence isomorphism above.
\end{proof}

Assume that $n\geq6$ is even and put $m=n-1$. Now
$\alpha=(1\ 2\ \cdots\ m)$ fixes $n$, so $g(n)$ is constant on every
coset in $\mathcal A_n$. Partition the vertices into
$
 W=\{g\la\alpha\ra:g(n)=n\}$ and $
 U=\{g\la\alpha\ra:g(n)\neq n\}.
$

Every coset $C\in U$ has a unique representative $g_C$ satisfying
\begin{equation}\label{eq:even-standard}
 g_C(m)=n.
\end{equation}
Indeed, the value $n$ occurs at a unique position in $[m]$, and right
multiplication by a unique power of $\alpha$ moves that position to $m$.

Delete domain position $m$ and value $n$, and relabel the remaining domain
position $n$ as $m$. Define $\bar g\in S_m$ by
\begin{equation}\label{eq:gbar}
 \bar g(i)=g(i)\quad(1\leq i\leq m-1),\qquad
 \bar g(m)=g(n).
\end{equation}
The permutations $\bar g$ form one parity class $\Omega$ of $S_m$. Deleting
position $m$ and value $n$ changes the sign by the constant factor
$(-1)^{m+n}=-1$, and every member of that parity class has a unique even
extension satisfying \eqref{eq:even-standard}. Thus
\begin{equation}\label{eq:even-vertex-map}
 \Phi:U\longrightarrow\Omega,\qquad
 \Phi(C)=\bar g_C,
\end{equation}
is a bijection.

For $1\leq i\leq m-2$, put
$
 \delta_i=\alpha^i\rho\alpha^{-i}=(i+1\ n\ i).
$
The permutation $\delta_i$ fixes $m$, so it preserves the standard
condition \eqref{eq:even-standard}. The hyperedge
$g\alpha^i\la\rho\ra$ is incident with the vertices represented by
$g,g\delta_i,g\delta_i^2$. Under \eqref{eq:gbar}, $\delta_i$ induces a
$3$-cycle on $\{i,i+1,m\}$.

This hyperedge correspondence is injective. If
$
 g\alpha^i\la\rho\ra=h\alpha^j\la\rho\ra
$
for standard representatives $g,h$ and $1\leq i,j\leq m-2$, then
$
 h=g\alpha^i\rho^t\alpha^{-j}
$
for some $t\in\{0,1,2\}$. Since $g(m)=h(m)=n$, the permutation
$u=\alpha^i\rho^t\alpha^{-j}$ fixes $m$. But
$
 \alpha^{-j}(m)=m-j\in\{2,\ldots,m-1\},
$
which is fixed by $\rho^t$. Hence $u(m)=\alpha^i(m-j)=m$ forces $i=j$,
and then $h=g\delta_i^t$. The abstract hyperedges are equal.

It follows that the permutation hypergraph generated by
\begin{equation}\label{eq:fan}
 \{1,2,m\},\{2,3,m\},\ldots,\{m-2,m-1,m\}
\end{equation}
is isomorphic to a spanning subhypergraph of $\mathcal H_n[U]$. These
triangles form a fan $2$-tree, and the vertex $1$ is simplicial. Fix
$s\in[m]$, say $s=1$, and let
\begin{equation}\label{eq:R}
 R=\{C\in U:\bar g_C(1)=s\}.
\end{equation}
By Corollary~\ref{cor:layer-deletion}, the induced hypergraph on
$U\setminus R$ contains a spanning loose hyperforest $F$ with two
components.

Define a bipartite multigraph $B$ with parts $R$ and $W$. Each
$\la\rho\ra$-coset incident with one vertex of $R$ and one vertex of $W$
gives an edge between those vertices. Parallel edges are retained.

\begin{lemma}\label{lem:B-local}
Every edge of $B$ comes from a hyperedge of $\mathcal H_n$ containing one
vertex in each of $R$, $W$, and $U\setminus R$. Every vertex of $R$ has
degree $2$ in $B$.
\end{lemma}

\begin{proof}
Let $C=g\la\alpha\ra\in R$, where $g$ is the standard representative.
Then
$
 g(m)=n, g(1)=s $ and $g(n)\neq n.
$
The hyperedges incident with $C$ are represented uniquely by
$g\alpha^k$, $0\leq k\leq m-1$, by
Lemma~\ref{lem:trivial-intersection}. The three incident
$\la\alpha\ra$-cosets have $n$-images
\begin{equation}\label{eq:n-images}
 g(n),\qquad g(\alpha^k(m)),\qquad g(\alpha^k(1)).
\end{equation}
The first is not $n$. Since $g(m)=n$, one of the other two is $n$ exactly
when $k=0$ or $k=m-1$. Hence $C$ is incident with two hyperedges meeting
$W$.

For $k=0$, the $W$-vertex is $g\rho\la\alpha\ra$. The other $U$-vertex is
$g\rho^2\la\alpha\ra$, whose standard representative is $g\rho^2\alpha$;
its first image is
$
 (g\rho^2\alpha)(1)=g(2)\neq g(1)=s.
$
It lies in $U\setminus R$. For $k=m-1$, put $x=g\alpha^{-1}$. The
$W$-vertex is $x\rho^2\la\alpha\ra$, and the other $U$-vertex is
$x\rho\la\alpha\ra$. Here $x\rho$ is standard and
$
 (x\rho)(1)=g(n)\neq g(1)=s.
$
This also lies in $U\setminus R$. The description shows that a counted
hyperedge cannot contain two vertices of $R$.
\end{proof}

\begin{lemma}\label{lem:B-regular}
The bipartite multigraph $B$ is $2$-regular and therefore has a perfect
matching.
\end{lemma}

\begin{proof}
By Lemma~\ref{lem:B-local}, every vertex in $R$ has degree $2$. Let
$
 G_0=\{h\in A_n:h(n)=n,\ h(s)=s\}.
$
Left multiplication by $G_0$ preserves $R$, $W$, and hyperedge incidence,
so it acts by automorphisms of $B$. This action is transitive on $W$.
Indeed, for $C,D\in W$, choose the unique representatives $g\in C$ and
$h\in D$ satisfying $g(m)=h(m)=s$. The even permutation $hg^{-1}$ fixes
$n$ and $s$ and maps $C$ to $D$. Hence all vertices of $W$ have the same
degree.

Moreover,
$
 |R|=\frac{(m-1)!}{2}=|W|.
$
The first equality is the size of a layer of $\Omega$, while
$|W|=|A_m|/m=(m-1)!/2$. Equality of total degrees on the two sides shows
that every vertex in $W$ also has degree $2$. Each component of a finite
$2$-regular bipartite multigraph is an even cycle, with parallel edges
allowed as a $2$-cycle. Alternating edges form a perfect matching.
\end{proof}

\begin{proposition}\label{prop:even-forest}
If $n\geq6$ is even, then $\mathcal H_n$ contains a spanning loose
hyperforest with two components.
\end{proposition}

\begin{proof}
Start with the two-component hyperforest $F$ spanning $U\setminus R$. Let
$M$ be a perfect matching of $B$. For each edge of $M$, add the corresponding
hyperedge of $\mathcal H_n$. By Lemma~\ref{lem:B-local}, it contains one
previously present vertex of $U\setminus R$ and two new vertices, one in $R$
and one in $W$. The matching covers every vertex of $R\cup W$ exactly once.
Thus the added hyperedges may be processed in any order; each meets the
preceding vertex set in one vertex and introduces two new vertices. Neither
component of $F$ is merged with the other, so the result is spanning and
still has two components.
\end{proof}
If \(n=4\), choose the least hyperedge of \(\mathcal H_4\)
under the fixed ordering of cosets, and leave the remaining
vertex isolated.

If \(n\ge5\) is odd, use the identification in (12), take the identity
as the root, and apply Theorem~4.2 recursively to the \(2\)-path (14).

If \(n\ge6\) is even, use (17), take \(s=1\), apply Corollary~4.3 to
the fan (18), and construct \(B\).

The case $n=4$ can be handled directly.

\begin{lemma}\label{lem:n4}
The hypergraph $\mathcal H_4$ has a spanning loose hyperforest with two
components.
\end{lemma}

\begin{proof}
Here $|A_4|=12$ and $|\la\alpha\ra|=3$, so $\mathcal H_4$ has four
vertices. By Lemma~\ref{lem:trivial-intersection}, every hyperedge contains
three distinct vertices. Take one hyperedge and leave the fourth vertex
isolated.
\end{proof}

\begin{proof}[Proof of Theorem~\ref{thm:main}]
For $n=4$, combine Lemma~\ref{lem:n4} with
Theorem~\ref{thm:reduction}. For odd $n\geq5$, use
Proposition~\ref{prop:odd-forest}; for even $n\geq6$, use
Proposition~\ref{prop:even-forest}. These cases cover every $n\geq4$.
\end{proof}

Together with Proposition~\ref{prop:parity}, the theorem gives the exact
cycle-path distinction.

\begin{corollary}\label{cor:cycle-path}
For every $n\geq4$, the digraph $P(n,n-2)$ has a directed Hamiltonian path
but no directed Hamiltonian cycle.
\end{corollary}

We finish by stating the constructive content of the proof. Order
permutations by one-line notation and cosets by their least elements. For odd
$n$, use the identification in \eqref{eq:odd-vertex-map}, take the identity
as the root, and apply Theorem~\ref{thm:root-deleted} recursively to the
$2$-path \eqref{eq:2path}, using the choices in
Remark~\ref{rem:deterministic}. For even $n$, use
\eqref{eq:even-vertex-map}, take $s=1$, apply
Corollary~\ref{cor:layer-deletion} to the fan \eqref{eq:fan}, and construct
$B$. Every component of $B$ has two alternating perfect matchings; choose
the one containing the least edge in that component. This produces the
spanning two-component hyperforest.

The number of selected hyperedges is
\begin{equation}\label{eq:forest-size}
 \frac{q_n-2}{2},
\end{equation}
where $q_n$ is given in \eqref{eq:qn}. Start with $s(g)=g\alpha$ and switch
the selected hyperedges in loose-hyperforest order. The result has two
cycles. Choose the least $\la\rho\ra$-coset meeting both hyperforest
components and, within it, the least ordered pair of elements in different
successor cycles. Exchange their targets as in Lemma~\ref{lem:splice}.
Delete the unique invalid transition, traverse the remaining successor
relation from its head to its tail, and replace every $g\in A_n$ by
$(g(1),\ldots,g(n-2))$.

\begin{proposition}\label{prop:algorithm}
The preceding deterministic procedure terminates for every $n\geq4$ and
outputs a directed Hamiltonian path in $P(n,n-2)$.
\end{proposition}

\begin{proof}
For \(n=4\), the construction is given explicitly by Lemma~7.1.
Assume henceforth that \(n\ge5\). 
Every recursive call deletes one vertex of the underlying $2$-tree, so the
hyperforest construction terminates. Propositions~\ref{prop:odd-forest} and
\ref{prop:even-forest} establish the required hyperforest. Its switches give
two cycles by Proposition~\ref{prop:forest-cover}. A crossing hyperedge
exists by Lemma~\ref{lem:H-connected}, and the splice and cut are valid by
Lemma~\ref{lem:splice} and Theorem~\ref{thm:reduction}. Every tie-breaking
rule selects from a nonempty finite set.
\end{proof}


\begin{thebibliography}{99}

\bibitem{Brunat1997}
J.M. Brunat, M.A. Fiol, M.L. Fiol,
Digraphs on permutations,
\emph{Discrete Math.} 174 (1997) 73--86.

\bibitem{Chang2025}
Z. Chang, L. Diao, S. Wang,
Efficient methods of constructing universal cycles for $k$-permutations,
\emph{Discrete Appl. Math.} 374 (2025) 120--134.

\bibitem{Chang2026}
Z. Chang, L. Diao,
Efficient methods of constructing shorthand universal cycles for permutations,
\emph{Discrete Appl. Math.} 386 (2026) 217--227.

\bibitem{Chung1992}
F.R.K. Chung, P. Diaconis, R.L. Graham,
Universal cycles for combinatorial structures,
\emph{Discrete Math.} 110 (1992) 43--59.

\bibitem{Curren1996}
S.J. Curren, J.A. Gallian,
Hamiltonian cycles in Cayley graphs and digraphs---a survey,
\emph{Discrete Math.} 156 (1996) 1--18.

\bibitem{Gao2019}
A.L.L. Gao, S. Kitaev, W. Steiner, P.B. Zhang,
On a greedy algorithm to construct universal cycles for permutations,
\emph{Internat. J. Found. Comput. Sci.} 30 (2019) 61--72.

\bibitem{Gabric2020}
D. Gabri\'c, J. Sawada, A. Williams, D.C.H. Wong,
A successor rule framework for constructing $k$-ary de Bruijn sequences and
universal cycles,
\emph{IEEE Trans. Inform. Theory} 66 (2020) 679--687.

\bibitem{Holroyd2012}
A.E. Holroyd, F. Ruskey, A. Williams,
Shorthand universal cycles for permutations,
\emph{Algorithmica} 64 (2012) 215--245.

\bibitem{Isaak2006}
G. Isaak,
Hamiltonicity of digraphs for universal cycles of permutations,
\emph{European J. Combin.} 27 (2006) 801--805.

\bibitem{Jackson1993}
B.W. Jackson,
Universal cycles of $k$-subsets and $k$-permutations,
\emph{Discrete Math.} 117 (1993) 141--150.

\bibitem{Jackson2009}
B. Jackson, B. Stevens, G. Hurlbert,
Research problems on Gray codes and universal cycles,
\emph{Discrete Math.} 309 (2009) 5341--5348.

\bibitem{Johnson2009}
J.R. Johnson,
Universal cycles for permutations,
\emph{Discrete Math.} 309 (2009) 5264--5270.

\bibitem{Rankin1948}
R.A. Rankin,
A campanological problem in group theory,
\emph{Proc. Cambridge Philos. Soc.} 40 (1948) 17--25.

\bibitem{Ruskey2010}
F. Ruskey, A. Williams,
An explicit universal cycle for the $(n-1)$-permutations of an $n$-set,
\emph{ACM Trans. Algorithms} 6 (2010), Article 45, 12 pp.

\bibitem{Sawada2023}
J. Sawada, J. Sears, A. Trautrim, A. Williams,
Concatenation trees: A framework for efficient universal cycle and de Bruijn
sequence constructions,
arXiv:2308.12405 (2023).

\bibitem{Starling2003}
A.G. Starling, J.B. Klerlein, J. Kier, E.C. Carr,
Cycles in the digraph $P(n,k)$: an algorithm,
\emph{Congr. Numer.} 162 (2003) 129--137.

\bibitem{Swan1999}
R.G. Swan,
A simple proof of Rankin's campanological theorem,
\emph{Amer. Math. Monthly} 106 (1999) 159--161.

\bibitem{Witte1984}
D. Witte, J.A. Gallian,
A survey: Hamiltonian cycles in Cayley graphs,
\emph{Discrete Math.} 51 (1984) 293--304.

\bibitem{Wong2017}
D.C.H. Wong,
A new universal cycle for permutations,
\emph{Graphs Combin.} 33 (2017) 1393--1399.

\end{thebibliography}
\end{document}